\documentclass[12pt]{amsart}

\usepackage{hyperref}
\hypersetup{
	bookmarks=true,         
	unicode=false,          
	pdftoolbar=true,        
	pdfmenubar=true,        
	pdffitwindow=false,     
	pdftitle={},    
	pdfauthor={},     
	pdfsubject={},   
	pdfcreator={Me},   
	pdfproducer={},  
	pdfkeywords={}, 
	pdfnewwindow=true,      
	colorlinks=false,       
	linkcolor=red,          
	citecolor=green,        
	filecolor=magenta,      
	urlcolor=cyan           
}

\usepackage[dvipsnames]{xcolor}
\usepackage{amssymb,verbatim}
\usepackage{amsmath,amsfonts,enumitem}
\usepackage[mathscr]{euscript} 
\usepackage{amsthm}
\usepackage{url}
\usepackage{graphicx} 
\usepackage{enumitem} 
\usepackage{hyperref} 
\hypersetup{colorlinks} 
\usepackage{bbm} 
\usepackage{bm} 
\usepackage{mathrsfs} 
\usepackage{cancel} 

\usepackage[normalem]{ulem}

\usepackage[colorinlistoftodos]{todonotes}

\RequirePackage{cleveref}
\usepackage{hypcap}
\hypersetup{colorlinks=true, citecolor=darkblue, linkcolor=darkblue}
\definecolor{darkblue}{rgb}{0.0,0,0.7}
\newcommand{\darkblue}{\color{darkblue}}

\definecolor{darkred}{rgb}{0.68,0,0}

\definecolor{darkgreen}{rgb}{0,.38,0}

\newcommand{\defn}[1]{\emph{\darkblue #1}}

\setlist[enumerate]{
	label=\textnormal{({\roman*})},
	ref={\roman*}}

\makeatletter
\def\th@plain{%
	\thm@notefont{}
	\itshape 
}
\def\th@definition{%
	\thm@notefont{}
	\normalfont 
}
\makeatother

\newtheorem{thm}{Theorem}
\newtheorem{lemma}[thm]{Lemma}

\newtheorem*{claim*}{Claim}

\newtheorem{conv}[thm]{Convention}

\theoremstyle{definition}
\newtheorem{ex}[thm]{Example}

\numberwithin{figure}{section}
\numberwithin{equation}{section}

\def\<{\langle}
\def\>{\rangle}

\def\der{\text{{\rm D}}}
\def\contr{\text{{\rm C}}}
\def\subd{\text{{\rm S}}}

\def\0{{\mathbf 0}}

\def\.{\hskip.06cm}

\def\.{\hskip.06cm}

\newcounter{sideremark}

\title{A higher-dimensional version of Fáry's Theorem}

\author[Karim Adiprasito]{Karim Adiprasito}
\author[Zuzana Patáková]{Zuzana Patáková}
\address{Sorbonne Université and Université Paris Cité, CNRS, IMJ-PRG, F-75005 Paris, France}
\email{karim.adiprasito@imj-prg.fr}
\address{Department of Algebra, Faculty of Mathematics and Physics, Charles University, Prague, Czech Republic}
\email{patakova@karlin.mff.cuni.cz}

\thanks{\today}

\begin{document}
	
	\maketitle

		\begin{abstract}
			We prove a generalization of István Fáry's celebrated theorem to higher dimensions. Namely, we show that if a finite simplicial complex $X$ can be piecewise linearly embedded into a $d$-dimensional PL manifold $M$, then there is a triangulation of $M$ containing $X$ as a subcomplex.
		\end{abstract}

		\bigskip
		
		\bigskip

		Fáry's theorem is simple: 
		
		\begin{thm}[\cite{fary}]
			Any simple, planar graph can be drawn without crossings in the plane so that its edges are straight line segments.
		\end{thm}

		That is simpler and more beautiful than rarely possible. So, could this be true in greater generality? Initially, the answer is no. There are many complexes that have an embedding into an Euclidean space, but not one that is affine on every simplex \cite{AKM23, BS92}: In other words, when embedding $k$-dimensional complexes into $\mathbb{R}^d$, and $d\le 2k$, it can happen that the complex embeds, but does not embed linearly. In other dimensions, as one says, the linear and piecewise linear picture coincide, but only for a trivial reason: Every simplicial complex of dimension $k$ has a geometric realization in $\mathbb{R}^d$ if $d$ exceeds $2k$.
		
		
		So, paper done, right? There is no positive theorem, apparently.
		
		Except, there is: 
		We prove the following:
		
		\begin{thm}[Kind-of Fáry's theorem]\label{t:main}
			Consider $X$ a finite simplicial complex, and a piecewise linear embedding $\varphi:X\longrightarrow M$, where $M$ is a PL $d$-manifold. Then there is a triangulation $M'$ of $M$ that contains
			$X$ as a subcomplex.
		\end{thm}

  Throughout the paper we consider finite simplicial complexes, and refer the reader to Zeeman \cite{Zee63} for an introduction to PL topology.
			
		Some people may consider this theorem as a folklore, but as knowledge of PL topology is gradually lost, we feel it is useful to record it here. We note that it is necessary to modify the embedding map $\varphi:X\longrightarrow M$ to achieve this.

		\begin{ex}
			Consider a 2-dimensional simplex $\sigma$, that we wish to embed into the 4-sphere $S^4$ as follows:
			Assume $S^4$ is the suspension of a triangulated $S^3$, with apex points $n$ and $s$. Consider $\gamma$ a knot in $S^3$, and the 2-dimensional disk $n\ast \gamma$, where $\ast$ stands for a join. 
			
			Choose now a "bad" homeomorphism $\varphi$ sending $(\sigma, \partial \sigma)$ to the pair $(n\ast \gamma, \gamma)$ so that in particular some interior point of $\sigma$ is mapped to $n$. Obviously, every point of the image $\varphi(\sigma^\circ)$ apart from $n$ has a flat neighborhood with respect to $\sigma$, that is, a neighborhood where the embedding is homeomorphic to the standard embedding of $\mathbb{R}^2$ into $\mathbb{R}^4$.
		\end{ex}
		
		Let us first note a lemma of Bing that almost solves the problem:
		\begin{lemma}[Section I\.2\cite{Bing83}]\label{lem:bing}
			Consider $\widetilde{X}$ a simplicial complex, and a simplexwise linear embedding $\varphi:\widetilde{X}\longrightarrow M$, where $M$ is a PL $d$-manifold. Then there is a triangulation $\widetilde{M}$ of $M$ that contains $\widetilde{X}$ as a subcomplex.
		\end{lemma}

		Bing only proved this in the case of $M$ being a $3$-manifold, but his proof extends naturally to the case of arbitrary dimension. In order not to rely on this, we note that it is enough that some subdivision of $\widetilde{X}$ extends to a triangulation of $M$. 

\begin{lemma}[Chapter 1, Lemma 4 \cite{Zee63}]\label{lem:Zeeman}
			Consider $\widetilde{X}$ a simplicial complex, and a simplexwise linear embedding $\varphi:\widetilde{X}\longrightarrow M$, where $M$ is a PL $d$-manifold. Then there is a triangulation $\widetilde{M}$ of $M$ that contains some subdivision of $\widetilde{X}$ as a subcomplex.
\end{lemma}

		\section{The proof}
		\subsection{Subdivisions.}
		Recall that a \defn{stellar subdivision} of a simplicial complex $\Delta$ at a simplex $\sigma$ is the result of removing all simplices of $N_\sigma \Delta$, the collection of simplices in $\Delta$ containing $\sigma$, and attaching the cone $v_\sigma\ast \partial N_\sigma \Delta$ to it along the natural identification map. 
		
		Recall that the \defn{derived subdivision} of a simplicial complex $\Delta$ is the simplicial complex $\der\Delta$ of all order chains of simplices (or equivalently, applying stellar subdivisions at all faces in reverse order of inclusion\footnote{I.e.\ we first subdivide all facets, then all the faces of codimension 1, etc.}).
		
		Given two simplicial complexes $\Gamma \subseteq \Delta$, we say that $\Gamma$ is an \defn{induced subcomplex} of $\Delta$ if every simplex in $\Delta$ with all vertices in $\Gamma$ is a simplex in $\Gamma$ as well. It is easy to see that
		if $\Gamma \subseteq \Delta$, then the derived subdivision of $\Gamma$ is an induced subcomplex of the derived subdivision\footnote{Derived subdivisions are ``compatible'' in the following sense: the derived subdivision of $\Gamma$ coincides with the derived subdivision of $\Delta$ restricted to $\Gamma$.} of $\Delta$, see Figure \ref{f:induced} for illustration.

		\begin{figure}[hbt]
		 \includegraphics[page=1]{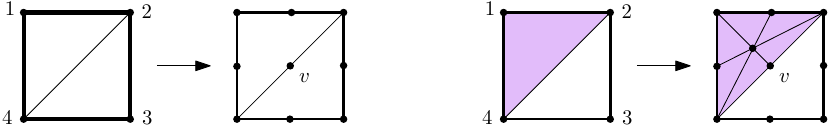}
		 \caption{Derived subdivision $\der\Gamma$ as an induced subcomplex of the derived subdivision $\der\Delta$. Example shows a 4-cycle $\Gamma = \{12, 23, 34, 14\}$ (we list only the top dimensional faces) and $\Delta = \{12, 23, 34, 14, 24\}$ on the left, $\Delta = \{124, 23, 34\}$ on the right, respectively. \\ The vertex $v$ certifies that the final complex is not strongly induced.}\label{f:induced}
		\end{figure}

Consider now the \defn{biased derived subdivision} of a pair of simplicial complexes $\Gamma \subseteq \Delta$: It is obtained by applying stellar subdivisions at faces of $\Delta$ not in $\Gamma$, in order of reverse inclusion; for illustration see Figure \ref{f:strongly_induced} (left). In the following lemma we relate it to strongly induced subcomplexes: A subcomplex $\Gamma$ of a simplicial complex $\Delta$ is a  \defn{strongly induced subcomplex} of $\Delta$ if for every simplex $\sigma$ in $\Delta \setminus \Gamma$, $\Gamma \cap \mathrm{st}_\sigma\, \Delta$ consists of a single simplex, where $\mathrm{st}_\sigma\, \Delta$ is the \defn{star} of $\sigma$ in $\Delta$, that is, the simplicial closure of all faces of $\Delta$ containing $\sigma$. Note that strongly induced subcomplex is automatically an induced subcomplex, the opposite implication does not hold as demonstrated on Figure \ref{f:induced}.

  \begin{lemma}\label{lem:stronginduced}
  Let $\Gamma$ be an induced subcomplex of $\Delta$ and let $\Delta'$ be the biased derived subdivision of the pair $\Gamma \subseteq \Delta$. Then $\Gamma$ is a strongly induced subcomplex of $\Delta'$.
  \end{lemma}

\begin{proof}
Let $\sigma$ be a simplex of $\Delta'$ not in $\Gamma$.  We want to show that 
there is a simplex $\tau$ of $\Gamma$ such that \[ \tau\ =\ \Gamma \cap \mathrm{st}_\sigma\, \Delta'.\]

The vertices of $\Delta'$ correspond to faces of $\Delta$ not in $\Gamma$, and vertices of $\Gamma$. 
Hence, seeing $\sigma$ as a union of its vertices, it consists of the disjoint union of a face $\hat{\sigma}$ of $\Gamma$ and a chain $\sigma_1 < \sigma_2 < \dots < \sigma_k$ of faces in $\Delta$ not in $\Gamma$, where $\hat{\sigma} < \sigma_1$.  The intersection of $\mathrm{st}_\sigma\, \Delta'$ with $\Gamma$ coincides with the intersection $\sigma_1 \cap \Gamma \subseteq \Delta$. Since $\Gamma$ is induced in $\Delta$, the latter is a face.
See  Figure \ref{f:strongly_induced} for illustration.
  \begin{figure}[hbt]
        \includegraphics[page=2]{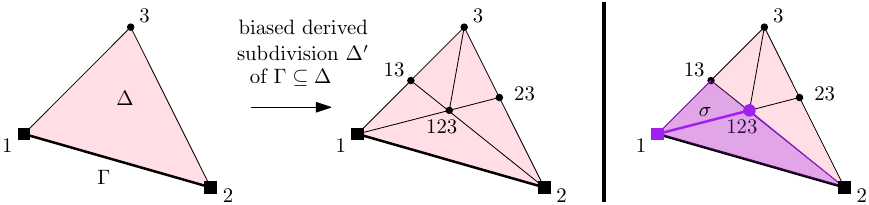}
        \caption{On the left: Biased derived subdivision $\Delta'$ of the pair $\Gamma \subseteq \Delta$, where $\Delta$ is a 2-simplex and $\Gamma$ is the edge spanned by vertices 1 and 2. On the right: Let $\Gamma \subseteq \Delta$ be as on the left and let $\sigma = \{1, 123\}$ be the highlighted edge of $\Delta'$. Then $\hat \sigma = \{1\}, \sigma_1 = \{123\}$, $\mathrm{st}_\sigma\, \Delta'$ is the highlighted subcomplex and $\Gamma \cap \mathrm{st}_\sigma\, \Delta' = \Gamma \cap \sigma_1 = 12$.}\label{f:strongly_induced}
        \end{figure}
\end{proof}

 \subsection{Edge contractions \& subdivisions}     
		Let $\Gamma$ be a simplicial complex. An \defn{edge contraction} is the operation of removing the neighborhood $N_e \Gamma$ of the edge $e=\{0,1\}$ and identifying $N_0\partial N_e \Gamma$ and $N_1\partial N_e \Gamma$; we denote the result by $\contr_e\Gamma$. We say that the edge $e$ is \defn{valid} if $\contr_e\Gamma$ is a simplicial complex and call it a \defn{valid edge contraction}. Notice that the edge $e$ is valid if and only if $e$ is not contained in a \defn{missing simplex} of $\Gamma$, by which we mean a simplex which is not contained in $\Gamma$ but whose all proper faces are. See Figure \ref{f:subd_contr} for illustration.
		
		An \defn{edge subdivision}, denoted by $\subd_e\Gamma$, is a stellar subdivision at the edge $e$.
		Let $\Gamma$ be a (induced) subcomplex of $\Delta$ and  let $e$ be an edge in $\Gamma$, then $\subd_e\Gamma$ is a (induced) subcomplex of $\subd_e\Delta$. See Figure \ref{f:subd_contr} for illustration.

		\begin{figure}[hbt]
		 \includegraphics[page=3]{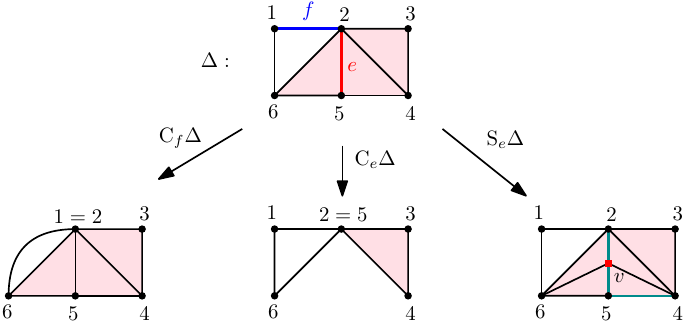}
		 \caption{Example of an edge contraction and an edge subdivision. Edge $f$ is not valid as it is contained in a missing simplex determined by vertices $1,2$ and $6$. Edge $e$ is valid in $\Delta$ and $\contr_e\Delta$ is a simplicial complex. Let $\Gamma$ be a subcomplex of $\Delta$ with top dimensional faces $25, 45$. Then $\subd_e\Gamma$ is a subcomplex of $\subd_e\Delta$, but it is not induced as certified by vertices $v, 4, 5$.}
		 \label{f:subd_contr}
		\end{figure}

		Now we inspect the relation between valid edge contraction and a strong inducedness.
  
  \begin{lemma}\label{lem:edge_contr}
  Let  $\Gamma$ be a strongly induced subcomplex of $\Delta$ and let $e$ be a valid edge of $\Gamma$. Then $\contr_e\Gamma$ is a strongly induced subcomplex of $\contr_e\Delta$. In particular, $\contr_e\Delta$ is a simplicial complex.
  \end{lemma}

  \begin{proof}
  First we show that $\contr_e \Delta$ is a simplicial complex. We need to argue that $e$ is a valid edge in $\Delta$, which follows from the obvious fact that missing simplices of $\Delta$ intersecting $\Gamma$ in more than a single vertex are missing simplices of $\Gamma$. 
  
  Clearly, $\contr_e\Gamma$ is a subcomplex of $\contr_e\Delta$.
  Let us now assume for contradiction that $\contr_e\Gamma$ is not strongly induced in $\contr_e\Delta$. That is, there exists a simplex $\sigma$ in $\contr_e\Delta\setminus \contr_e\Gamma$ for which $\contr_e\Gamma \cap \mathrm{st}_\sigma\, \contr_e\Delta$ is more than a single face. We may assume that $\sigma$ does not intersect $\contr_e\Gamma$. 
  Since $\contr_e\Gamma \cap \mathrm{st}_\sigma\, \contr_e\Delta$ is the image of $\Gamma \cap \mathrm{st}_\sigma\, \Delta$ under the contraction operation, $\Gamma \cap \mathrm{st}_\sigma\, \Delta$ is not a single simplex either.
  \end{proof}

Let us now look at  the relation between edge subdivision and a strong inducedness.
  
\begin{lemma}\label{lem:edge_subd} Let $\Gamma$ be a strongly induced subcomplex of $\Delta$ and let $e$ be an edge of $\Gamma$. Then $\subd_e\Gamma$ is a strongly induced subcomplex of $(\subd_e\Delta)'$, where  $'$ denotes the biased derived subdivision of the pair $\subd_e\Gamma \subseteq \subd_e\Delta$.
\end{lemma}

\begin{proof}
Since $\subd_e\Gamma$ is an induced subcomplex of $\subd_e\Delta$, we can apply Lemma~\ref{lem:stronginduced}.
\end{proof}

\begin{proof}[Proof of Theorem \ref{t:main}]	
Let us note that because $X$ embeds piecewise linearly,  there is a refinement $\widetilde{X}$ of $X$ that embeds simplexwise linearly. Let $\widetilde{M}$ denote the simplicial triangulation of $M$ output by the Bing Lemma \ref{lem:bing} (or alternatively the Zeeman Lemma \ref{lem:Zeeman}), i.e.\ we have that the subdivision $\widetilde X$ of $X$ is a subcomplex of $\widetilde M$.
		
Consider now the simplicial complex $\mathrm{sd}\, \widetilde{M}$, which  is the biased derived subdivision of the pair $\der\widetilde X \subseteq \der\widetilde M$,
where $\der$ denotes the operation of derived subdivision.
In words,  we get $\mathrm{sd}\, \widetilde{M}$
by first applying the derived subdivision to $\widetilde{M}$ and $\widetilde{X}$ (this makes the derived subdivision $\der\widetilde X$  an induced subcomplex $\der\widetilde M$), and then performing the biased derived subdivision at the obtained pair. By Lemma \ref{lem:stronginduced}, $\der \widetilde X$ is a strongly induced subcomplex of $\mathrm{sd}\, \widetilde{M}$.		

Recall that by results of Alexander \cite[Corrolary 10:2d]{Ale30}, and Newman \cite{New31}, two PL homeomorphic complexes are related by edge subdivisions and valid edge contractions.
		
Hence, we apply a suitable sequence of edge subdivisions and valid edge contractions of edges of $\der\widetilde X \subseteq \mathrm{sd}\, \widetilde{M}$ to transform $\der\widetilde{X}$ into $X$. Lemmata \ref{lem:edge_contr} and \ref{lem:edge_subd} allow us to extend the edge contractions and edge subdivisions to the whole complex ($\mathrm{sd}\, \widetilde{M}$ at the first step), while the edge subdivision is followed by the respective biased derived subdivision as explained in Lemma \ref{lem:edge_subd}.  It follows that the resulting triangulation $M'$ of $M$ is the desired one. 

Note that Lemmata \ref{lem:edge_contr} and \ref{lem:edge_subd} 
guarantee that $M'$ is indeed a triangulation, meaning that it stays a simplicial complex in each step during the process (in order to extend the edge contraction from a subcomplex to the whole complex we need the subcomplex to be strongly induced, see Lemma \ref{lem:edge_contr}). Note also that the biased derived subdivisions following edge contractions used in Lemma \ref{lem:edge_subd} do not change the respective subcomplex, so they are irrelevant with respect to the sequence of edge subdivisions and valid edge contractions we need to perform (by the result of Alexander and Newman), however they are obviously relevant for the triangulation.
\end{proof}

		\subsection*{Acknowledgments}
		We thank Misha Gromov and Grigori Avramidi for asking us this question.
		
		The first author is supported by the
		Centre National de Recherche Scientifique, and the Horizon Europe ERC Grant number:
		101045750 / Project acronym: HodgeGeoComb. The second author is supported by the GA\v CR grant no. 22-19073S.
		

\begin{thebibliography}{ABCDE1}

                \bibitem[AKM23]{AKM23} M. Abrahamsen, L. Kleist, and T. Miltzow. Geometric Embeddability of Complexes is $\exists \mathbb{R}$-Complete. \emph{39th International Symposium on Computational Geometry (SoCG 2023)}. Leibniz International Proceedings in Informatics (LIPIcs), Volume 258, pp. 1:1-1:19
				
				\bibitem[Ale30]{Ale30}
				J.~W.~Alexander,
				The combinatorial theory of complexes,
				\emph{Annals of Math.}~\textbf{31} (1930), 292--320.
				
				
				
				\bibitem[Bing83]{Bing83}
				R.~H.~Bing,
				\emph{The geometric topology of $3$-manifolds}, AMS, Providence, RI, 1983, 238~pp.
				
				\bibitem[BS92]{BS92} U. Brehm and K. S. Sarkaria. Linear vs. Piecewise-Linear Embeddability of Simplicial
				complexes. \emph{Technical Report MPI Bonn (1992)}. available at \url{https://archive.mpim-bonn.mpg.de/id/eprint/1946/1/preprint_1992_52.pdf}, pp. 1-14
				
				\bibitem[Fáry48]{fary}
	I. F\'{a}ry.
\newblock On straight line representation of planar graphs.
\newblock {\em Acta Univ. Szeged. Sect. Sci. Math.}, 11:229--233, 1948.

				\bibitem[New31]{New31} M. H. Alexander Newman, A theorem in combinatorial topology, \emph{J. London Math. Soc.}
				\textbf{6} (1931) 186–192
				
						\bibitem[Zee63]{Zee63}
			E.~C.~Zeeman,
			\emph{Seminar on combinatorial topology},
			Paris, IH\'ES, 1963, Ch.~1-6.
				
			\end{thebibliography}
		
		{\footnotesize

			\vskip.5cm
		}

	\end{document}